\documentclass{amsart}
\usepackage{amsmath}
\usepackage{graphicx}
\usepackage{latexsym}
\usepackage{pinlabel}
\usepackage{verbatim}
\usepackage[all]{xy}

\newtheorem{thm}{Theorem}

\newtheorem{cor}[thm]{Corollary}

\newtheorem{theorem}[thm]{Theorem}

\theoremstyle{definition}

\newtheorem{remark}[thm]{Remark}

\newtheorem{problem}[thm]{Problem}

\def\D{\mathcal D}

\newcommand{\CPb}{\overline{\mathbb{CP}}{}^{2}}
\newcommand{\CP}{{\mathbb{CP}}{}^{2}}

\newcommand{\Z}{\mathbb{Z}}
\newcommand{\N}{\mathbb{N}}

\newcommand{\K}{{\rm K3}}

\def \x {\times}
\def \eu{{\text{e}}}

\newcommand{\nc}{\newcommand}
\nc{\dmo}{\DeclareMathOperator}

\dmo{\MCG}{Mod}
\dmo{\Diff}{Diff}

\begin{document}

\title[Minimality and fiber sum decompositions of Lefschetz fibrations]
{Minimality and fiber sum decompositions of Lefschetz fibrations}

\author[R. \.{I}. Baykur]{R. \.{I}nan\c{c} Baykur}
\address{Department of Mathematics and Statistics, University of Massachusetts, Amherst, MA 01003, USA}
\email{baykur@math.umass.edu}

\begin{abstract}
We give a short proof of a conjecture of Stipsicz on the minimality of fiber sums of Lefschetz fibrations, which was proved earlier by Usher. We then construct the first examples of genus $ g \geq 2$ Lefschetz fibrations on minimal
symplectic 4-manifolds which, up to diffeomorphisms of the summands, admit unique decompositions as fiber sums.
\end{abstract}


\maketitle

\setcounter{secnumdepth}{2}
\setcounter{section}{0}

\section{Introduction} 

Despite the spectacular advancement in our understanding of the topology of symplectic $4$-manifolds in the past couple of decades, we still lack even a surgical classification scheme. As the generalized fiber sum operation can be performed in the symplectic realm, symplectic $4$-manifolds can be thought as sums of indecomposable ones. Although tracking this sort of data is usually very difficult, certain smooth information, such as minimality and Kodaira dimension of a symplectic\linebreak $4$-manifold is not out of reach \cite{Us1, Us2}. A more tractable set-up is when the symplectic $4$-manifold is equipped with a Lefschetz fibration (and all can be, after blow-ups, by Donaldson's seminal work on the existence of Lefschetz pencils on symplectic $4$-manifolds), where the decomposition is restricted to the standard fiber sum of Lefschetz fibrations: in this case, the sum always takes place along fibers, and translates to a combinatorial problem of expressing a positive Dehn twist factorization of the identity in the mapping class group as a product of such subwords. This extra structure indeed allows one to stably classify symplectic $4$-manifolds equipped with Lefschetz fibrations \cite{Au2, EndoEtal}. The purpose of this note is to further explore the properties of fiber sum decompositions of symplectic $4$-manifolds equipped with Lefschetz fibrations. 

We call a Lefschetz fibration \textit{indecomposable}, if it cannot be expressed as a fiber sum of two nontrivial Lefschetz fibrations. Indecomposable Lefschetz fibrations were first detected by Stipsicz in \cite{St1}, who proved that a relatively minimal Lefschetz fibration over the $2$-sphere admitting a $(-1)$-section is fiber sum indecomposable (also see \cite{Sm1}). Motivated by this, Stipsicz conjectured that in general a non-minimal $4$-manifold cannot be a fiber sum of two non-trivial (i.e. with at least one singular fiber) relatively minimal Lefschetz fibrations \cite[Conjecture~2.3]{St1}. This conjecture was later proved by Usher \cite{Us1}, by a thorough analysis of the effect of the fiber sum operation on the Gromov invariants of the symplectic $4$-manifold, where he made extensive use of relative Gromov invariants and gluings results due to Ionel and Parker. Relying on the additional Lefschetz fibration structure, we will provide a shorter and simpler proof of the same result, which is obtained in the same spirit as Stipsicz's original approach in the presence of $(-1)$-sections \cite{St1} and that of Sato's in the presence of $(-1)$-bisections of genus $\geq 2$ Lefschetz fibrations \cite{Sa1}. This is given in Theorem~\ref{minimality} below. A corollary we obtain is a recap of another result of Usher \cite{Us2} in this setting, which can be formulated as \textit{``fiber sums of Lefschetz fibrations produce only symplectic $4$-manifolds with positive Kodaira dimension''}; see Corollary~\ref{ruled} below. These results are proved in Section~3. 

Clearly, any Lefschetz fibration is a fiber sum of indecomposable ones, possibly with only one summand. In comparison with other prime decomposition theorems in low dimensional topology, it is natural to ask for which Lefschetz fibrations the summands, as \textit{smooth} $4$-manifolds, are uniquely determined. This is known to hold for genus $1$ Lefschetz fibrations, in which case, the elliptic surface $E(n)$ decomposes exactly into $n$ summands of $E(1)$. However, such a \textit{unique} prime decomposition result for Lefschetz fibrations, even up to diffeomorphisms of the summands (as much as it would have been wonderful to have it!), fails to hold once $g \geq 2$; in fact, many examples with multiple decompositions can be produced by looking at the stable equivalence of Lefschetz fibrations under fiber sums with certain universal fibrations; see \cite{EndoEtal}.

Nevertheless, as reviewed above, Lefschetz fibrations on non-minimal $4$-manifolds provide vacuous examples of uniquely decomposing genus $g \geq 2$ fibrations. To the best of our knowledge, no such fibrations with more than one summand were known, nor were there examples of uniquely decomposing ones on minimal Lefschetz fibrations. Our second main theorem, Theorem~\ref{unique}, shows that for almost all $g \geq 1$ (except possibly for a few small values of $g$), and for any given $m \geq 2$, there are infinite families of genus $g$ Lefschetz on minimal $4$-manifolds which decompose uniquely into $m$ indecomposable summands, up to diffeomorphisms. (When the fiber genus $g>1$, the uniqueness is achieved for given $m$, i.e., these can possibly decompose into other numbers of indecomposable summands.) Our proof of Theorem~\ref{unique} will be based on Stipsicz's analysis of Lefschetz fibrations on ruled surfaces in \cite{St2}, and gives rise to a couple of interesting questions we pose at the end of Section~4. In the Appendix, we will revise and correct the proof of the main theorem of \cite{St2} we used here.

\vspace{0.1in}
\noindent \textit{Acknowledgements.}  The author was partially supported by the NSF Grant DMS-1510395 and the Simons Foundation Grant 317732.



\vspace{0.1in}
\section{Preliminaries} 

In this article we will always work with closed, smooth, oriented $4$-manifolds. A \emph{genus $g$ Lefschetz fibration} $(X,f)$ is a surjection $f$ from a $4$-manifold $X$ onto $S^2$ that is a submersion on the complement of finitely many points $p_i$, at which there are local complex coordinates (compatible with the orientations on $X$ and $S^2$) with respect to which the map takes the form $(z_1, z_2) \mapsto z_1 z_2$, where the genus of a regular fiber $F$ is $g$. We will moreover assume that all $p_i$ lie in distinct fibers; this can always be achieved after a small perturbation. We call a fibration \emph{nontrivial} if it has at least one critical point, and \emph{relatively minimal}, if there are no spheres of self-intersection $-1$ contained in the fibers. By the Gompf-Thurston construction, total spaces of nontrivial Lefschetz fibrations always admit symplectic forms with respect to which all regular fibers are symplectic. 

A widely used way of constructing new Lefschetz fibrations from given ones is the \emph{fiber sum} operation, defined as follows: Let $(X_i, f_i)$, $i=1,2$, be genus $g$ Lefschetz fibrations with regular fiber $F$. The \emph{fiber sum} \, $(X_1, f_1)\#_{F, \Phi}(X_2,f_2)$ is a genus $g$ Lefschetz fibration obtained by removing a fibered tubular neighborhood of a regular fiber from each $(X_i, f_i)$ and then identifying the resulting boundaries via a fiber-preserving, orientation-reversing diffeomorphism $\Phi$. We drop $\Phi$ from the notation whenever the gluing is made so that the fibers are identified by the identity map on $F$, which is often called an \emph{untwisted} fiber sum. A Lefschetz fibration $(X,f)$ is called \emph{indecomposable} if it cannot be expressed as a fiber sum of any two nontrivial Lefschetz fibrations. Such fibrations can be regarded as prime building blocks of Lefschetz fibrations.

Lastly, let us review the notion of \emph{symplectic Kodaira dimension} we will repeatedly refer to in our discussions. First, recall that a symplectic $4$-manifold $(X, \omega)$ is called \emph{minimal} if it does not contain any embedded symplectic sphere of square $-1$, and that it can always be blown-down to a minimal symplectic\linebreak $4$-manifold $(X_{\text{min}}, \omega')$. Let $\kappa_{X_{\text{min}}}$ be the canonical class of $(X_{\text{min}}, \omega_{\text{min}})$. We then define the symplectic Kodaira dimension of $(X, \omega)$, denoted by $\kappa=\kappa(X,\omega)$ as 
\[
\kappa(X,\omega)=\left\{\begin{array}{rl}-\infty& \mbox{if
}\kappa_{X_{\text{min}}}\cdot[\omega_{\text{min}}]<0 \mbox{ or } \kappa_{X_{\text{min}}}^{2}<0 \\
0 & \mbox{if } \kappa_{X_{\text{min}}}\cdot[\omega_{\text{min}}]=\kappa_{X_{\text{min}}}^{2}=0\\ 1 &
\mbox{if }\kappa_{X_{\text{min}}}\cdot[\omega_{\text{min}}]>0\mbox{ and
}\kappa_{X_{\text{min}}}^{2}=0\\2& \mbox{if }\kappa_{X_{\text{min}}}\cdot[\omega_{\text{min}}]>0\mbox{
and }\kappa_{X_{\text{min}}}^{2}>0\end{array}\right.
\]
Here $\kappa$ is independent of the minimal model $(X_{\text{min}}, \omega_{\text{min}})$ and is a smooth invariant of the $4$-manifold $X$. 
 
The reader can turn to \cite{GS} for more on Lefschetz fibrations and to \cite{Li3} for the symplectic Kodaira dimension.

\vspace{0.1in}
\section{Minimality and fiber sum decompositions} 

Here we prove Stipsicz's conjecture on the minimality of fiber sums of relatively minimal Lefschetz fibrations. 

\begin{theorem} \label{minimality}
Fiber sum of two nontrivial relatively minimal Lefschetz fibrations is a minimal $4$-manifold.
\end{theorem}

\begin{proof}
Let $X$ be a $4$-manifold which is a fiber sum of two nontrivial relatively minimal genus $g \geq 1$ Lefschetz fibrations, $(X,f)= (X_1, f_1) \, \#_{F, \Phi} \, (X_2, f_2)$. By the nontriviality of the summands, the fiber $F$ of $(X,f)$ is homologically essential even if the fiber genus is $1$, and we can equip $X$ with a Thurston-Gompf symplectic form $\omega$ which makes the fibers symplectic. Moreover, we can choose an $\omega$-compatible almost complex structure $J$, even a generic one in the sense of Taubes (see e.g \cite{UsherDS}), with respect to which $f$ is $J$-holomorphic (for a suitable choice of almost complex structure on the base $2$-sphere).

If $X$ is not minimal, it follows from Taubes' seminal work on the correspondence between Gromov and Seiberg-Witten invariants on symplectic $4$-manifolds with $b^+>1$ \cite{Ta, Ta2} that any smooth $(-1)$-sphere is homologous to one that is $J$-holomorphic  \cite[Theorem~3.3]{St3}. The same holds when $b^+(X)=1$, possibly after changing the orientation of the $(-1)$-sphere, provided its pairing with the canonical class is $\pm 1$ \cite[Theorem~A]{Li1}. (Note that the symplectic representative in \cite[Theorem~A]{Li1} is obtained as a $J$-holomorphic curve with nontrivial Gromov-Taubes invariant.) 

It follows that we have a $J$-holomorphic $(-1)$-sphere $S$ in $X$, and $f|_S: S \to S^2$ is a $J$-holomorhic covering map. Thus $S$ is a (branched) multisection of $(X,f)$, i.e. it intersects all but finitely many fibers positively at exactly $n=S \cdot F$ points.

By isotoping $S$ if necessary, we can assume that $S$ decomposes as $S=S_1 \cup_{B} S_2$ where each $S_i$ is a (possibly disconnected) multisection of $f_i|_{X_i \setminus \nu(F)}$ with all the branched points in the interior. Thus, $B = \partial S_1 = - \partial S_2$ is a disjoint union of $n$ circles. Below, we will use the short-hand notation $D_k$ for a $2$-sphere with $k \geq 1$ disks removed, where $\Sigma_g$ denotes the closed orientable surface of genus $g$. Let
\[ S_1 = D_{k_1} \sqcup \ldots \sqcup D_{k_{r_1}} \, \text{and} \, \, S_2 =  D_{l_1} \sqcup \ldots \sqcup D_{l_{r_2}} \, , \]
where $n = \sum_{i=1}^{r_1} k_i = \Sigma_{j=1}^{r_2} l_j$. 

Now take the untwisted fiber sum $(\D X_i, \D f_i)=(X_i,f_i) \, \#_F \, (X_i, f_i)$. Then $\D X_i$ is a symplectic $4$-manifold with $b^+>1$ for each $i=1,2$, as $(X_i,f_i)$ are assumed to be nontrivial \cite{St1}. In each ``double'' $\D X_i$, we also get a double of $S_i$, which, by the decompositions above, yield
\[ \D S_1= \D D_{k_1} \sqcup \ldots \sqcup \D D_{k_{r_1}}  = \Sigma_{k_1-1} \sqcup \ldots \sqcup \Sigma_{k_{r_1}-1}, \, \text{and} \, , \]
\[ \D S_2 = \D D_{l_1} \sqcup \ldots \sqcup \D D_{l_{r_2}} =  \Sigma_{l_1-1} \sqcup \ldots \sqcup \Sigma_{l_{r_2}-1} \, . \]
By the Seiberg-Witten adjunction inequality, the self-intersection of each one of the $\Sigma_g$ with $g\geq 1$ component is bounded above by $2g-2$. On the other hand, the self-intersection of each sphere component $\Sigma_0$ is bounded above by $-1$, since any homologically essential sphere (which is the case here, as each $\Sigma_0$ intersects the fiber at one point) in a symplectic $4$-manifold with $b^+>1$ has negative self-intersection. Moreover, each $\Sigma_{k_i-1} = \D D_{k_i}$ has self-intersection twice the number of that of $D_{k_i}$ (rel its boundary), and the same goes for each $\Sigma_{l_j-1}$. In particular, each sphere should have self-intersection $\leq -2$.

Hence, we arrive at the inequality
\[-2 = (\D S_1)^2 + (\D S_2)^2 \leq \sum_{i=1}^{r_1} (2(k_i-1) - 2) + \sum_{j=1}^{r_2} (2(l_j-1) - 2) \, , \]
implying
\[ -1 \leq \sum_{i=1}^{r_1} k_i + \sum_{j=1}^{r_2} l_j - 2 (r_1+ r_2) = 2 n - 2 (r_1+r_2). \]
So we get $n  \geq r_1 + r_2 $. However, $B$, which is a disjoint union of $n$ circles, splits the $2$-sphere $S$ into $n+1= r_1 + r_2$ components. The contradiction implies that $X$ could not admit such a smooth $(-1)$-sphere $S$. 

\end{proof}

What follows is a restriction on the symplectic topology of $4$-manifolds that arise as fiber sums of Lefschetz fibrations:

\begin{cor} \label{ruled}
A symplectic $4$-manifold which is a fiber sum of two nontrivial relatively minimal Lefschetz fibrations cannot be a rational or a ruled surface, nor can have a torsion canonical class (i.e. cannot have $\kappa \leq 0$), with the sole exception of the $\K$ surface. 
\end{cor}

\begin{proof}
By the above theorem, $X$ is minimal. The only \textit{minimal} rational or ruled surfaces are $\CP$ and $S^2$-bundles over $\Sigma_h$, for $h \geq 0$. Since the fiber of a nontrivial genus $g \geq 1$ Lefschetz fibration is a homologically essential self-intersection zero class in $H_2(X)$, it is an elementary observation that $\CP$ does not admit any Lefschetz \textit{fibrations} at all. 

Now let a Lefschetz fibration on the minimal ruled surface $X$ be a fiber sum of $(X_i,f_i)$, $i=1,2$. Any relatively minimal Lefschetz fibration satisfies $c_1^2(X_i) \,  \geq 4-4g$ \cite{St4}. So the equation
\[ c_1^2(X)= c_1^2(X_1)+c_1^2(X_2)+8g-8 \,  \]
combined with this inequality implies that $8-8h \geq 0$, which can only hold when $h \leq 1$. Moreover, assuming $g \geq 2$ (where it is obvious for $g=1$), we should have $c_1^2(X_1)=c_1^2(X_2) = 4g-4$, which is only possible when each $X_i$ is a ruled surface \cite{Li2}. Since signature is additive for fiber sum, each $X_i$ should be minimal as well. Let us now consider this remaining case. 

The fiber class of a Lefschetz fibration on a ruled surface can only be $m  R_0$, for $R_0$ self-intersection zero section of the degree $0$ ruling on $\Sigma_h \x S^2$, whereas it is $n(2R_1+S)$ for $R_1$ self-intersection $-1$ section and $S$ a fiber of the degree $1$ ruling on $\Sigma_h \widetilde{\x} S^2$. Note that in the former case $m S$ is not an option, since by tubing between $m$ copies of the sphere fiber we would obtain a genus $0$ representative in the same homology class of positive genus fiber $F$, which, being a symplectic surface, should minimize the genus in its homology class. By the same argument we see that $h \neq 0$ for $\Sigma_h \x S^2$, which is already the case for $\Sigma_h \widetilde{\x} S^2$ due to minimality. This leaves $h=1$ as the only possiblity.

Since the symplectic structure on a minimal ruled surface is unique up to deformations and symplectomorphisms \cite[Theorem~B]{Li1}, we can apply the adjunction formula and derive 
\[ g(F)=1+m(h-1), \text{and} \, \, g(F)= 1 + 2n(h-1) \, , \]
respectively. In either case, we get $g(F)=1$ for $h=1$, but no minimal ruled surface is the total space of a 
nontrivial genus $1$ Lefschetz fibration. This concludes that $X$ cannot be rational or ruled.

Lastly, if $X$ is a minimal symplectic $4$-manifold with torsion canonical class, by the adjunction \textsl{equality} we have $2g-2 = \eu(F) = F^2 + K \cdot F = 0$, which implies that $g=1$. By the classification of genus $g=1$ Lefschetz fibrations, this is only possible if $X = \K$. 
\end{proof}

\vspace{0.1in}
\section{Lefschetz fibrations with unique decompositions}

We now prove our second main result of the paper: 

\begin{theorem} \label{unique}
For any $g$, except possibly for $g=$\small{${2,3,4,5,7,9,11,13}$}, and any $m \geq 2$, there are relatively minimal genus $g$ Lefschetz fibrations on infinitely many minimal \linebreak $4$-manifolds $X_m(k)$, $k \in N$, each of which decomposes uniquely into $m$ indecomposable summands $X_0(k)$ up to diffeomorphism.
\end{theorem}

\begin{proof}
The result is classical for $g=1$, and henceforth we will assume $g \geq 2$. By Stipsicz's work in \cite{St2}, the minimum number of singular fibers of an even genus $g \geq 6$ (resp. odd genus $g \geq 15$) Lefschetz fibration on a symplectic $4$-manifold with $b^+=1$ is attained by a Lefschetz fibration with $2g+4$ (resp. $2g+10$) \emph{only} on the ruled surface $\Sigma_{g/2} \x S^2 \# 4 \CPb$ (resp. $\Sigma_{(g-1)/2} \x S^2 \# 8 \CPb$). This uniqueness phenomenon in $b^+=1$ case is what we will exploit below, and should explain the excluded values for $g$ in the statement of our theorem. 

Depending on the parity of $g$, let $(X_0,f_0)$ denote the relatively minimal genus $g$ Lefschetz fibration, with $2g+4$ (resp. $2g+10$) singular fibers when $g$ is even (resp. odd) on the above ruled surfaces. Explicit monodromies of such fibrations are obtained by Korkmaz in \cite{Korkmaz}, from which all we need is the following: Let $K=g/2$ for even $g$ and $(g-1)/2$ for odd $g$. It is easily seen that $H_1(X_0)$ is freely generated by the homology classes of curves $\{a_i, b_i \, | \, i = 1, \ldots, K \}$ on the fiber $F \cong \Sigma_g$, where $a_i, b_i$ are standard generators of $\pi_1(F)$. (These generators are given in Figure~4 of \cite{Korkmaz}.) Let $\phi_k$ be the self-diffeomorphism of $F$ given by $t_{a_1}^k$, for any $k \geq 0$. We set $(X_m(k),f_m(k))$ to be the untwisted fiber sum of $(m-1)$ copies of $(X_0, f_0)$ and a twisted sum with one copy of $(X_0,f_0)$, where the twisted boundary gluing is given by the fiber-preserving, orientation-reversing diffeomorphism $\Phi_k = \phi_k \x \text{conj}$ on $F \x S^1$. A straightforward calculation shows that $H_1(X_m(k)) = \Z^{K-1} \oplus (\Z / k \Z)$, so for each fixed $m$ and $g$, we have an infinite family of pairwise non-homotopic $4$-manifolds $X_m(k)$, for varying $k \in \N$. 

The rest of our arguments will work for any $m \geq 1$, $k \geq 0$, and $g$ as in the statement, so let us drop the extra decorations and simply continue with\linebreak $(X,f)=(X_m(k),f_m(k))$. We claim that $(X,f)$ bears all the properties we listed in the theorem. By Theorem~\ref{minimality}, $X$ is minimal, and $X_0$ is indecomposable. Thus, per our construction, $(X,f)$ does decompose as a fiber sum of $m$ indecomposable Lefschetz fibrations, which are all copies of $(X_0,f_0)$. We are left with showing that this is a unique decomposition.

First, note that $c_1^2$ of any fiber sum of Lefschetz fibrations $(X_i, f_i)$, for\linebreak $i=1, \ldots, m$, calculates as \,  $\sum_{i=1}^m c_1^2(X_i) + 8(m-1)(g-1)$, which can be seen by standard Euler characteristic and signature calculations for fiber sums. Applying this to $(X,f)$ which is fiber sum of $m$ copies of $(X_0,f_0)$, we get 
\[ c_1^2(X)= m \, c_1^2(X_0) + 8(m-1)(g-1)= 4m(1-g) + 8(m-1)(g-1) \]
no matter whether $g$ is even and odd. 

Now, assume that $(X,f)$ can be expressed as a fiber sum of $m$ fibrations $(X_i,f_i)$, $i=1, \ldots, m$, which are not necessarily indecomposable (i.e. there are possibly more than $m$ indecomposable summands). By Li's main theorem in \cite{Li2}, we have the inequality $c_1^2(X_i) \geq 2(1-g)$ when $X_i$ is not ruled, and as observed by Stipsicz \cite{St4}, we have $c_1^2(X_i) \geq 4(1-g)$ in general. It follows that, unless all $X_i$ are ruled, we have
\[ (m-1) 4(1-g) + 2(1-g)  \leq \, \sum_{i=1}^n c_1^2(X_i) = c_1^2(X) - 8(m-1)(g-1) \, . \]
From the calculation $c_1^2(X) = 4m(1-g) + 8(m-1)(g-1)$, we get 
\[ (m-1) 4(1-g) + 2(1-g) \leq 4m(1-g) \, , \]
which now implies $g \leq 1$. As we assumed $g\geq 2$, all $X_i$ should be ruled surfaces. 


Hence, we see that $(X,f)$ can only be written as a sum of exactly $m$ Lefschetz fibrations on ruled surfaces $(X_i,f_i)$. Since it has $m$ times the minimum number of singular fibers allowed on a ruled surface, each $(X_i, f_i)$
should attain the minimum number of singular fibers possible, which, in turn, shows that $X_i= \Sigma_{g/2} \x S^2 \# 4 \CPb$ when $g$ is even and $\Sigma_{(g-1)/2} \x S^2 \# 8 \CPb$ when it is odd. This completes the proof of the theorem.
\end{proof}

\begin{remark}
In the classical $g=1$ case, the proof of the unique factorization hinges on the rather simple structure of the genus $1$ mapping class group: one can classify all possible monodromy factorizations of $g=1$ Lefschetz fibrations up to Hurwitz equivalence, and derive the uniqueness of the summands up to diffeomorphism from this. There is little to no chance of implementing a similar proof when $g>1$ however, due to the far richer structure of the corresponding mapping class group. In the proof of Theorem~\ref{unique}, we overcame this difficulty by taking an alternate approach that heavily depends on symplectic geometry and Gauge theory.
\end{remark}

\begin{remark} \label{gap}
For the missing $g$ values not covered in Theorem~\ref{unique}, we can modify our proof by employing an indeterminate nontrivial Lefschetz fibration on a ruled surface with minimal number of singular fibers as one of the summands of $(X,f)$. However, to strike all the essential points in the proof such as making sure that for any fiber sum decomposition of $(X,f)$ the summands $X_i$ are ruled, one would still need to constrain $g$, only allowing a couple additions to our list in the statement, while losing the explicit nature of our construction. 
\end{remark}

\begin{remark} \label{KnotSurgery}
Refining our choice of fiber sum gluings (and arguments to follow), we can furthermore obtain an infinite family of Lefschetz fibrations by fiber summing two standard fibrations $(X_0,f_0)$ so that the resulting $4$-manifolds are pairwise homeomorphic but not diffeomorphic. Using fibered knots with same genus $g'$ but different Alexander polynomials, one can get relatively minimal genus $n-1+2g'$ Lefschez fibrations on knot surgered elliptic surfaces $E(n)$, which are examples of this kind \cite{FS}. For $E(2)= \K$, these are in fact twisted fiber sums of the genus $g$ Lefschetz fibrations on ruled surfaces $X_0= \Sigma_{g/2} \x S^2 \# 4 \CPb$ we have used in our proof above. 
\end{remark}

The examples we produced by self-sums of $(X_0,f_0)$ in the proof of Theorem~\ref{unique} mimic the case of elliptic fibrations. Reflecting on the challenges with detecting the uniqueness of decompositions, the first question we have is:

\begin{problem}
Find uniquely decomposing genus $g > 2$ Lefschetz fibrations which are not self-sums of the same fibration.
\end{problem}
\vspace{-0.2cm}
\noindent Here a unique summand should be understood as a trivial self-sum. 

\begin{remark}
It seems possible to find some sporadic examples when the fibration decomposes into $m=2$ indecomposable summands. When $g=2$, we can find some examples \textit{with reducible fibers} that are not self-sums and yet admit unique decompositions. Here our arguments rely heavily on the hyperellipticity of the genus-$2$ mapping class group. As in our proof of the above theorem, one still needs to appeal to deeper results from gauge theory and symplectic geometry to argue unique decomposability.
 
Before we present an example, let us introduce some practical notation: let $n$ denote the number of nonseparating and $s$ denote the number of separating vanishing cycles in a given genus-$2$ Lefschetz fibration $(X,f)$, and in this case, call $(X,f)$ of \textit{type $(n,s)$}. Note that this records the topology of fibers completely, since there is only one topological type of non-trivial separating cycle on a genus-$2$ surface. Since the first homology group of the mapping class group of the genus-$2$ mapping class group is $\Z_{10}$, where any Dehn twist along a nonseparating curve corresponds to $\bar{1}$ and any Dehn twist along the nontrivial separating curve corresponds to $\bar{2}$, we have the relation $n+2s  \equiv 0$ mod $10$. 

Since $8g-8+s \leq 5 n$ for any relatively minimal Lefschetz fibration (See the proof of Theorem~1.3 in \cite{St4}), we have $8+s \leq 5n$. Coupled with the congruence relation above, we conclude that the smallest number of critical points a nontrivial genus-$2$ Lefschetz fibration can have is $7$, which is of type $(4,3)$. An explicit monodromy of such a fibration $(X_1,f_1)$ is given in \cite{BaykurKorkmaz}. The second smallest, which is of type $(6,2)$, can be realized by Matsumoto's famous genus-$2$ Lefschetz fibration $(X_2,f_2)$. Letting $(X,f)$ be the (say, untwisted) fiber sum of the two, we clearly see that $(X,f)$ can only decompose into type $(4,3)$ and $(6,2)$ fibrations. Moreover, the summands are uniquely determined up to diffeomorphisms: any $X_1$ of type $(4,3)$ and any $X_2$ of type $(6,2)$ are diffeomorphic to $S^2 \x T^2$ blown-up $3$ times and $4$ times, respectively; see Proposition~$4.1$ \cite{Sa1}. (Note that, in contrast to Stipsicz's aforementioned result, here $b^+(X_1)=1$ and $f_1$ has less than $2g+4=8$ critical points.) Hence $(X,f)$ uniquely decomposes into $2$ unequal summands.
\end{remark}

\smallskip
Lefschetz fibrations convey different features depending on the symplectic Kodaira dimension of the underlying symplectic $4$-manifold. We see that there are uniquely decomposing Lefschetz fibrations for all Kodaira dimensions and for almost all $g \geq 2$: By Corollary~\ref{ruled}, rational and ruled surfaces ($\kappa = - \infty$), as well as symplectic $4$-manifolds with torsion canonical class ($\kappa = 0$) are indecomposable. As seen in the proof of Theorem~\ref{unique}, we have $c_1^2(X_m)=0$  for $m=2$ ($\kappa = 1$, as it cannot be $0$), and  $c_1^2(X_m(k))>0$  for $m>2$ ($\kappa= 2$). Here, the borderline $\kappa = 1$ case appears to be the most interesting. Note that Fintushel and Stern's Lefschetz fibrations on knot surgered elliptic surfaces discussed in Remark~\ref{KnotSurgery} provide a large family of examples with $\kappa=1$. We therefore ask:

\begin{problem}
Can a relatively minimal Lefschetz fibration $(X,f)$ with symplectic Kodaira dimension $\kappa = 1$ be decomposed into two summands in more than one way (up to diffeomorphisms of the summands)?
\end{problem}


\vspace{0.2in}
\refstepcounter{secnumdepth}
\refstepcounter{section}

\section*{Appendix: the minimum number of singular fibers on ruled surfaces}

Our proof of Theorem~\ref{unique} relied on the following results of Stipsicz from \cite{St2}: The minimum number of singular fibers of a nontrivial Lefschetz fibration on a $4$-manifold with $b^+=1$ is $2g+4$ for $g \geq 6$ and even, $2g+10$ for $g \geq 15$ and odd, and these bounds are sharp \cite[Theorem~1.1.(1)-(2)]{St2}. Moreover, in the course of the proof of this result \cite[Sections~4.1 and 4.2]{St2}, it is observed that these minimum values can be realized only on ruled surfaces, and these ruled surfaces are uniquely determined as $\Sigma_{g/2} \x S^2 \# 4\CPb$ and $\Sigma_{(g-1)/2} \x S^2 \# 8\CPb$, respectively. To eliminate some cases, these proofs make repeated use of \cite[Theorem~1.4]{St2} (also labeled as \cite[Theorem~2.9]{St2}), which states that the fiber class of a nontrivial Lefschetz fibration is primitive. However, there is a mistake in the proof of the latter:\footnote{Nevertheless, the statement would obviously hold whenever the fibration admits an honest section.} In the non-simply-connected case, it is claimed that the fiber class being primitive in a fiber sum of Lefschetz fibrations would imply that it is also primitive in the summands. However, this is not guaranteed due to the extra handles one gets from the regular neighborhood $\nu(F)$ when reconstructing $(X,f)$ from $X \setminus \nu(F)$. While we do not have a fix for this argument, we will show that all other results of \cite{St2} are correct, by employing arguments that replace the use of the problematic Theorem~1.4/2.9. Reassuring the reader of the validity of these results we have relied on in our article is the purpose of this Appendix. 

Henceforth, we follow the labeling in the article \cite{St2}. There are four instances that require our attention.  Theorem~1.4/2.9 is used in Lemma~4.4 and Lemma~4.7 to rule out certain cases where the fiber class would be non-primitive. In Remark~4.3, it is claimed that one can use the method described in Lemma~4.7 to prove that the minimal ruled surfaces $\Sigma_h \widetilde{\x} S^2$ and $\Sigma_h \x S^2$ do not admit genus $g=2h+2$ Lefschetz fibrations. Lastly, in the paragraph preceding Remark~4.8, it is claimed that one can use arguments similar to those in Lemma~4.4 to exclude the remaining cases. All these constitute parts of the proof of Theorem~1.1.

In the proof of Lemma~4.4, Theorem~2.9 is used to argue that the homology class of the fiber, which is $b$ times a sphere fiber of a ruled surface, is only possible when $b= \pm 1$. Instead, one can observe that such a class can be represented by tubing between $b$ disjoint copies of the sphere fiber, which therefore has a genus $0$ representative. This however contradicts the fact that the positive genus fiber, being a symplectic surface, should realize the minimum genus in its homology class. Moreover, the claim that precedes Remark~4.8 is now valid, provided we substitute the use of Theorem~2.9 with the argument above.

The proof of Lemma~4.7 is rather problematic, but in this case, we can invoke \cite[Proposition~4.4]{Li2}, which, by a degree argument, readily states that \emph{no} ruled surface over $\Sigma_h$ admits a genus $g$ Lefschetz fibration with $g < 2h$, covering what is claimed here and more. 

Remark~4.3 requires a bit more work. Recall from our proof of Corollary~\ref{ruled} in this paper that the fiber genus of a Lefschetz fibration on $X=\Sigma_h \x S^2$ is $g(F)=1+m(h-1)$, and on $X=\Sigma_h \widetilde{\x} S^2$ it is $g(F)= 1+2n(h-1)$. As no ruled surface admits a nontrivial genus $1$ Lefschetz fibration, we should have $m,n \geq 1$, $h \geq 2$. Since the number of singular fibers is $4(g-h)$, provided $m \geq 3$ or $n \geq 2$, the assumptions $g \geq 6$ for $g$ even and $g \geq 15$ for $g$ odd force the fibration to attain more than the proposed minimum number of singular fibers. The remaining $m=2$ or $n=1$ case means $g=2h-1$, which is ruled out by \cite[Proposition~4.4]{Li2} again.


\end{document}